\documentclass{amsart}
\usepackage{amsmath,amssymb,amsthm,bm,bbm}
\usepackage{graphicx,enumerate}
\usepackage{layout}
\usepackage{longtable}
\usepackage{url}

\theoremstyle{plain}
\newtheorem{theorem}{Theorem}[section]
\newtheorem*{theorem-nn}{Theorem}
\newtheorem{lemma}[theorem]{Lemma}

\newtheorem*{proposition-nn}{Proposition}
\newtheorem{corollary}[theorem]{Corollary}

\theoremstyle{definition}

\newtheorem{remark}[theorem]{Remark}

\theoremstyle{remark}

\newcommand{\bZ}{\mathbbm{Z}}\newcommand{\bQ}{\mathbbm{Q}}
\newcommand{\bG}{\mathbbm{G}}

\newcommand{\bP}{\mathbbm{P}}

\newcommand{\GL}{{\rm GL}}

\newcounter{sub}
{\begin{list}{(\arabic{sub})}{\usecounter{sub}%
\setlength{\leftmargin}{2em}}}{\end{list}}

\title[Rationality problem for norm one tori for dihedral extensions]
{Rationality problem for norm one tori for dihedral extensions}

\author[A. Hoshi]{Akinari Hoshi}
\address{Department of Mathematics, Niigata University, Niigata 950-2181, Japan}
\email{hoshi@math.sc.niigata-u.ac.jp}

\author[A. Yamasaki]{Aiichi Yamasaki}
\address{Department of Mathematics, Kyoto University, Kyoto 606-8502, Japan}
\email{aiichi.yamasaki@gmail.com}

\thanks{
This work was partially supported by JSPS KAKENHI Grant Numbers 
19K03418, 20H00115, 20K03511. 
}

\begin{document}
\begin{abstract}
We give a complete answer to the rationality problem 
(up to stable $k$-equivalence) 
for norm one tori $R^{(1)}_{K/k}(\bG_m)$ of $K/k$ 
whose Galois closures $L/k$ are dihedral extensions 
with the aid of Endo and Miyata \cite[Theorem 1.5, Theorem 2.3]{EM75} 
and Endo \cite[Theorem 2.1]{End11}. 
By using a similar technique, we give %alternative proofs 
refinements of the proof 
of stably rational cases of Endo and Miyata's theorems 
as an appendix of the paper. 
\end{abstract}

\maketitle

\tableofcontents

%%%%%%%%%%%%%%%%%%%%%%%%%%%%%%%%%%%%%%%%%%%%%%%%%%%%%%%%
\section{Introduction}\label{seInt}

Let $k$ be a field and $X$ be an integral $k$-variety 
of dimension $n$. 
The L\"uroth problem asks whether a $k$-unirational variety $X$ 
is $k$-rational, 
i.e. $X$ is $k$-birational to $\bP^n_k$ 
if there exists a dominant rational map $\bP^n_k\dashrightarrow X$. 
Over an algebraically closed field $k$ with char $k=0$, 
examples of $k$-unirational varieties which are not $k$-rational 
are given by 
Clemens and Griffiths \cite{CG72}, 
Iskovskikh and Manin \cite{IM71}, 
Artin and Mumford \cite{AM72}. 
See e.g. 
Manin \cite{Man86}, 
Manin and Tsfasman \cite{MT86}, 
Iskovskikh \cite{Isk97}, 
Colliot-Th\'{e}l\`{e}ne and Sansuc \cite[Section 1]{CTS07}, 
Beauville \cite{Bea16}, 
Voisin \cite{Voi16}. 
Over an algebraically non-closed field $k$, 
Swan \cite[Theorem 1]{Swa69} 
gave the first counterexamples to Noether's problem 
over $\bQ$ which are not $\bQ$-rational but $\bQ$-unirational varieties 
whose function fields are isomorphic to 
the field of invariants $\bQ(x_1,\ldots,x_p)^{C_p}$ where 
$C_p$ is the cyclic group of prime order $p$ with $p=47, 113, 233$. 
We see that $\bQ(x_1,\ldots,x_p)^{C_p}=L(M)^{C_{p-1}}(t)$ where 
$L(M)^{C_{p-1}}$ can be understood as the function field $\bQ(T)$ of 
an algebraic $\bQ$-torus $T$ of dimension $p-1$ via Galois descent where 
$C_{p-1}={\rm Gal}(L/\bQ)$ and $L=\bQ(\zeta_p)$ 
is the $p$-th cyclotomic field 
% (see Masuda \cite{Mas55}, \cite{Mas68}, 
% Swan \cite{Swa69}, Voskresenskii \cite{Vos70b}, \cite{Vos71}, \cite{Vos73}, 
% Endo and Miyata \cite{EM73}, Lenstra \cite{Len74}, Hoshi \cite{Hos15}, 
% Plans \cite{Pla17}). 
(see %also 
survey papers Swan \cite{Swa83} and Hoshi \cite[Section 2]{Hos20}). 
For algebraic group varieties $X$, 
examples of $k$-unirational varieties which are not $k$-rational 
are going back to Chevalley. 
Indeed, Chevalley \cite[Section V]{Che54} gave examples 
of norm one tori $R^{(1)}_{K/k}(\bG_m)$ of $K/k$ 
which are not $k$-rational but $k$-unirational 
where $k=\bQ_p$ is the field of $p$-adic numbers 
and $K/k$ are non-cyclic abelian extensions. 
Later, Endo and Miyata \cite{EM75} gave a necessary and sufficient 
condition for stably/retract rationality 
of norm one tori $R^{(1)}_{K/k}(\bG_m)$ 
of Galois extensions $K/k$ 
(see Theorem \ref{th2-2} and Theorem \ref{th2-3}). 
In particular, $R^{(1)}_{K/k}(\bG_m)$ is not (retract) $k$-rational 
but $k$-unirational 
if some $p$-Sylow subgroup of $G={\rm Gal}(K/k)$ is not cyclic. 

The Zariski problem asks whether a stably $k$-rational variety $X$ 
is $k$-rational, 
i.e. $X$ is $k$-birational to $\bP^n_k$ 
if $X\times_k \bP^m_k$ is $k$-birational to $\bP^{m+n}_k$. 
By using Galois lattice arguments (algebraic $k$-torus technique) on ${\rm Pic}\,\overline{X}$ 
attached to the dihedral group $D_n$ of order $2n$ $(n\geq 3:$ odd$)$, 
for an algebraically non-closed field $k$ (resp. an algebraically closed field $k$) 
with ${\rm char}\, k\neq 2$, 
examples of not $k$-rational but stably $k$-rational varieties $X$ with $X\times_k\bP^3$ $k$-rational 
are given by 
Beauville, Colliot-Th\'el\`ene, Sansuc and Swinnerton-Dyer \cite{BCTSSD85} 
as (generalized) Ch\^{a}telet surfaces given by $y^2-az^2=P(x)\in k[x]$ 
where ${\rm deg}_x P(x)=n\geq 3$ and ${\rm Gal}(P(x)/k)\simeq D_n$ with $n=3$ 
(resp. $3$-folds given by $y^2-a(t)z^2=P(t,x)\in k(t)[x]$ where 
${\rm deg}_x P(t,x)=3$). 
%
%Let $\overline{k}$ be a fixed separable closure of $k$ and $\mathcal{G}={\rm Gal}(\overline{k}/k)$.  
%
Indeed, 
stable $k$-rationality of $X$ is obtained by constructing 
a $k$-rational torsor $\mathcal{F}$ on $X$ under $S_0$  
%with $k$-rational $\mathcal{F}$ 
which is $k$-birational to 
$X\times_k S_0$ where $S_0$ is an algebraic $k$-torus with 
${\rm Hom}(S_0,\bG_m)\simeq {\rm Pic}\,\overline{X}$ for $n=3$ 
(see also Colliot-Th\'el\`ene and Sansuc \cite{CTS80}) and 
%$\mathcal{G}$-lattice
${\rm Pic}\, \overline{X}$ is stably permutation for odd $n\geq 3$ 
although non-rationality of $X$ needs results of 
Iskovskikh \cite{Isk67}, \cite{Isk70}, \cite{Isk72} over $k\neq\overline{k}$ 
(resp. Clemens and Griffiths \cite{CG72} and 
Beauville \cite{Bea77a}, \cite{Bea77b} over $k=\overline{k}$). 
See \cite{BCTSSD85} for details and also Shepherd-Barron \cite{She04}. 
% 
%Indeed, they showed that 
%for a proper, smooth, geometrically integral $k$-variety $X$ 
%which is $\overline{k}$-rational, 
%$X$ is stably $k$-rational if and only if 
%(i) $X$ has a universal torsor $\mathcal{F}$ which is $k$-rational and 
%(ii) $\mathcal{G}$-lattice ${\rm Pic}\, \overline{X}$ is stably permutation 
%(see \cite[Section 2]{BCTSSD85}). 
%
We also refer to 
Benoist and Wittenberg \cite{BW20}, \cite{BW23} which investigate 
Clemens-Griffiths method, 
intermediate Jacobians and rationality over non-closed field $k$. 

Let $K/k$ be a separable field extension of degree $m$ 
and $L/k$ be the Galois closure of $K/k$. 
Let $G={\rm Gal}(L/k)$ and $H={\rm Gal}(L/K)$ with $[G:H]=m$. 
The Galois group $G$ may be regarded as a transitive subgroup of 
the symmetric group $S_m$ of degree $m$. 
%%%
We may assume that 
$H$ is the stabilizer of one of the letters in $G\leq S_m$, 
i.e. $L=k(\theta_1,\ldots,\theta_m)$ and $K=L^H=k(\theta_i)$ 
where $1\leq i\leq m$. 
Then we have $\bigcap_{\sigma\in G} H^\sigma=\{1\}$ 
where $H^\sigma=\sigma^{-1}H\sigma$ and hence 
$H$ contains no normal subgroup of $G$ except for $\{1\}$. 

Let $T=R^{(1)}_{K/k}(\bG_m)$ be the norm one torus of $K/k$,
i.e. the kernel of the norm map $R_{K/k}(\bG_m)\rightarrow \bG_m$ where 
$R_{K/k}$ is the Weil restriction (see \cite[page 37, Section 3.12]{Vos98}). 
Note that $R_{K/k}(\bG_m)$ is $k$-rational with ${\rm dim}\, R_{K/k}(\bG_m)=[K:k]=m$ 
and 
${\rm dim}\, T=m-1$. 
Let $D_n$ be the dihedral group of order $2n$ $(n\geq 3)$ 
and $Z(D_n)$ be the center of $D_n$. 
In this paper, we investigate the rationality problem for 
$R^{(1)}_{K/k}(\bG_m)$ with $G\simeq D_n$. 
For $G\simeq D_n$, by the condition $\bigcap_{\sigma\in G} H^\sigma=\{1\}$, 
we obtain that $H=\{1\}$ or $|H|=2$ with $H\neq Z(D_n)$. 
When $H=\{1\}$, i.e. $K/k$ is Galois, 
Endo and Miyata \cite[Theorem 2.3, Theorem 1.5]{EM75} proved that 
$R^{(1)}_{K/k}(\bG_m)$ is stably $k$-rational 
(resp. not retract $k$-rational) if $n$ is odd (resp. even). 
When $|H|=2$ with $H\neq Z(D_n)$, 
Endo \cite[Theorem 3.1, Theorem 2.1]{End11} showed that 
$R^{(1)}_{K/k}(\bG_m)$ is stably $k$-rational 
(resp. not retract $k$-rational) if $n$ is odd (resp. $2$-power). 
See Section \ref{sePre} for details. %in particular, Example \ref{ex2-6} (2)). 

The main theorem of this paper 
gives a complete answer to the rationality problem 
(up to stable $k$-equivalence) 
for norm one tori $R^{(1)}_{K/k}(\bG_m)$ of $K/k$ 
whose Galois closures $L/k$ are dihedral extensions 
with the aid of Endo and Miyata \cite[Theorem 2.3, Theorem 1.5]{EM75} 
and Endo \cite[Theorem 3.1]{End11}: 
\begin{theorem}\label{mainth}
Let $K/k$ be a finite non-Galois separable field extension 
and $L/k$ be the Galois closure of $K/k$. 
Let $D_n=\langle x,y\mid x^n=y^2=1,y^{-1}xy=x^{-1}\rangle$ 
be the dihedral group of order $2n$ $(n\geq 4:$ even$)$ 
and $Z(D_n)=\langle x^{\frac{n}{2}}\rangle$ be the center of $D_n$. 
Assume that $G={\rm Gal}(L/k)\simeq D_n$  
and $H={\rm Gal}(L/K)\leq G$. 
Then $|H|=2$ with $H\neq Z(D_n)$ and we have:\\
{\rm (i)} if $n\equiv 0\ ({\rm mod}\ 4)$, then 
$R^{(1)}_{K/k}(\bG_m)$ is not retract $k$-rational;\\
{\rm (ii)} if $n\equiv 2\ ({\rm mod}\ 4)$, then 
$R^{(1)}_{K/k}(\bG_m)$ is stably $k$-rational. 
More precisely, 
$R^{(1)}_{K/k}(\bG_m)\times R_{k_1/k}(\bG_{m,k_1})\times 
R_{k_2/k}(\bG_{m,k_2})\times R_{k_3/k}(\bG_{m,k_3})$ 
of dimension $(n-1)+\frac{n}{2}+\frac{n}{2}+2=2n+1$ and 
$R_{k_4/k}(\bG_{m,k_4})\times R_{k_5/k}(\bG_{m,k_5})\times \bG_{m,k}$ 
of dimension $n+n+1=2n+1$ 
are birationally $k$-equivalent where 
$R_{k_i/k}(\bG_{m,k_i})$ are $k$-rational with 
$k_1=L^{\langle x^{\frac{n}{2}},x^2y\rangle}$, 
$k_2=L^{\langle x^{\frac{n}{2}},xy\rangle}$, 
$k_3=L^{\langle x\rangle}$, 
$k_4=L^{\langle x^{\frac{n}{2}}\rangle}$, 
$k_5=L^{\langle xy\rangle}$, 
$\langle x^{\frac{n}{2}},x^2y\rangle\simeq 
\langle x^{\frac{n}{2}},xy\rangle \simeq C_2\times C_2$, 
$\langle x\rangle\simeq C_n$, 
$\langle x^{\frac{n}{2}}\rangle\simeq \langle xy\rangle\simeq C_2$ 
and %with 
$[k_1:k]=[k_2:k]=\frac{n}{2}$, $[k_3:k]=2$, $[k_4:k]=[k_5:k]=n$. 
\end{theorem}

We organize this paper as follows.  
In Section \ref{sePre}, 
we prepare related materials and known results including some examples
in order to prove Theorem \ref{mainth}. 
In Section \ref{seProof}, 
the proof of Theorem \ref{mainth} is given. 
In Section \ref{seAlt}, 
we give 
%alternative proofs 
refinements of the proof of 
stably rational cases with 
$G\simeq D_n$ $(n\geq 3:$ odd$)$ of Endo and Miyata's theorems 
\cite[Theorem 2.3]{EM75} with $H=\{1\}$ and 
\cite[Theorem 3.1]{End11} with $H\simeq C_2$, 
using a technique similar to that given in the proof 
of Theorem \ref{mainth}. 

%%%%%%%%%%%%%%%%%%%%%%%%%%%%%%%%%%%%%%%%%%%%%%%%%%%%%%%%
\section{Preliminaries}\label{sePre}

Let $k$ be a field and $K$ %and $K^\prime$ 
be a finitely generated field extension of $k$. 
A field $K$ is called {\it rational over $k$} 
(or {\it $k$-rational} for short) 
if $K$ is purely transcendental over $k$, 
i.e. $K$ is isomorphic to $k(x_1,\ldots,x_n)$, 
the rational function field over $k$ with $n$ variables $x_1,\ldots,x_n$ 
for some integer $n$. 
$K$ is called {\it stably $k$-rational} 
if $K(y_1,\ldots,y_m)$ is $k$-rational for some algebraically 
independent elements $y_1,\ldots,y_m$ over $K$. 
Two fields 
$K$ and $K^\prime$ are called {\it stably $k$-isomorphic} if 
$K(y_1,\ldots,y_m)\simeq K^\prime(z_1,\ldots,z_n)$ over $k$ 
for some algebraically independent elements $y_1,\ldots,y_m$ over $K$ 
and $z_1,\ldots,z_n$ over $K^\prime$. 
When $k$ is an infinite field, 
%%%
$K$ is called {\it retract $k$-rational} 
if there is a $k$-algebra $R$ contained in $K$ such that 
(i) $K$ is the quotient field of $R$, and (ii) 
the identity map $1_R : R\rightarrow R$ factors through a localized 
polynomial ring over $k$, i.e. there is an element $f\in k[x_1,\ldots,x_n]$, 
which is the polynomial ring over $k$, and there are $k$-algebra 
homomorphisms $\varphi : R\rightarrow k[x_1,\ldots,x_n][1/f]$ 
and $\psi : k[x_1,\ldots,x_n][1/f]\rightarrow R$ satisfying 
$\psi\circ\varphi=1_R$ (cf. Saltman \cite{Sal84}). 
$K$ is called {\it $k$-unirational} 
if $k\subset K\subset k(x_1,\ldots,x_n)$ for some integer $n$. 
It is not difficult to see that
\begin{center}
$k$-rational \ \ $\Rightarrow$\ \ 
stably $k$-rational\ \ $\Rightarrow$ \ \ 
retract $k$-rational\ \ $\Rightarrow$ \ \ 
$k$-unirational.
\end{center} 

%%%%%%%%%%%%%%%
Let $\overline{k}$ be a fixed separable closure of the base field $k$. 
Let $T$ be an algebraic $k$-torus, 
i.e. a group $k$-scheme with fiber product (base change) 
$T\times_k \overline{k}=
T\times_{{\rm Spec}\, k}\,{\rm Spec}\, \overline{k}
\simeq (\bG_{m,\overline{k}})^n$; 
$k$-form of the split torus $(\bG_m)^n$. 
An algebraic $k$-torus 
$T$ is said to be {\it $k$-rational} (resp. {\it stably $k$-rational}, 
{\it retract $k$-rational}, {\it $k$-unirational}) 
if the function field $k(T)$ of $T$ over $k$ is $k$-rational 
(resp. stably $k$-rational, retract $k$-rational, $k$-unirational). 
Two algebraic $k$-tori $T$ and $T^\prime$ 
are said to be 
{\it birationally $k$-equivalent $(k$-isomorphic$)$} 
%(resp. {\it stably birationally $k$-equivalent}), 
%denoted by $T\approx T^\prime$ 
%(resp. $T\stackrel{\rm s.b.}{\approx} T^\prime$), 
if their function fields $k(T)$ and $k(T^\prime)$ are 
$k$-isomorphic. %(resp. stably $k$-equivalent). 
For an equivalent definition in the language of algebraic geometry, 
see e.g. 
Manin \cite{Man86}, 
Manin and Tsfasman \cite{MT86}, 
Colliot-Th\'{e}l\`{e}ne and Sansuc \cite[Section 1]{CTS07}, 
Beauville \cite{Bea16}, Merkurjev \cite[Section 3]{Mer17}. 

Let $L$ be a finite Galois extension of $k$ and $G={\rm Gal}(L/k)$ 
be the Galois group of the extension $L/k$. 
Let $M=\bigoplus_{1\leq i\leq n}\bZ\cdot u_i$ be a $G$-lattice with 
a $\bZ$-basis $\{u_1,\ldots,u_n\}$, 
i.e. finitely generated $\bZ[G]$-module 
which is $\bZ$-free as an abelian group. 
Let $G$ act on the rational function field $L(x_1,\ldots,x_n)$ 
over $L$ with $n$ variables $x_1,\ldots,x_n$ by 
\begin{align}
\sigma(x_i)=\prod_{j=1}^n x_j^{a_{i,j}},\quad 1\leq i\leq n\label{acts}
\end{align}
for any $\sigma\in G$, when $\sigma (u_i)=\sum_{j=1}^n a_{i,j} u_j$, 
$a_{i,j}\in\bZ$. 
The field $L(x_1,\ldots,x_n)$ with this action of $G$ will be denoted 
by $L(M)$.
%An algebraic torus $T$ over $k$ is a group $k$-scheme 
%such that $T\otimes_k \overline{k}\simeq \bG_{m,\overline{k}}^n$. 
There is the duality between the category of $G$-lattices 
and the category of algebraic $k$-tori which split over $L$ 
(see \cite[Section 1.2]{Ono61}, \cite[page 27, Example 6]{Vos98}). 
%A torus $T$ corresponds in this duality to the dual $X(T)^\circ$ 
%of the character group $X(T)={\rm Hom}(T,\bG_m)$. 
In fact, if $T$ is an algebraic $k$-torus, then the character 
group $\widehat{T}={\rm Hom}(T,\bG_m)$ of $T$ 
may be regarded as a $G$-lattice. 
Conversely, for a given $G$-lattice $M$, there exists an algebraic 
$k$-torus $T={\rm Spec}(L[M]^G)$ which splits over $L$ 
such that $\widehat{T}\simeq M$ as $G$-lattices. 

The invariant field $L(M)^G$ of $L(M)$ under the action of $G$ 
may be identified with the function field $k(T)$ 
of the algebraic $k$-torus $T$ where $\widehat{T}\simeq M$. 
Note that the field $L(M)^G$ is always $k$-unirational 
(see \cite[page 40, Example 21]{Vos98}). 
Isomorphism classes of tori of dimension $n$ over $k$ correspond bijectively 
to the elements of the set $H^1(\mathcal{G},\GL_n(\bZ))$ 
where $\mathcal{G}={\rm Gal}(\overline{k}/k)$ since 
${\rm Aut}(\bG_m^n)=\GL_n(\bZ)$. 
The $k$-torus $T$ of dimension $n$ is determined uniquely by the integral 
representation $h : \mathcal{G}\rightarrow \GL_n(\bZ)$ up to conjugacy, 
and the group $h(\mathcal{G})$ is a finite subgroup of $\GL_n(\bZ)$ 
(see \cite[page 57, Section 4.9]{Vos98})). 

%Let $M$ be a $G$-lattice. 
A $G$-lattice $M$ is called {\it permutation} $G$-lattice 
if $M$ has a $\bZ$-basis permuted by $G$, 
i.e. $M\simeq \oplus_{1\leq i\leq m}\bZ[G/H_i]$ 
for some subgroups $H_1,\ldots,H_m$ of $G$. 
$M$ is called {\it stably permutation} 
$G$-lattice if $M\oplus P\simeq P^\prime$ 
for some permutation $G$-lattices $P$ and $P^\prime$. 
$M$ is called {\it invertible} 
if it is a direct summand of a permutation $G$-lattice, 
i.e. $P\simeq M\oplus M^\prime$ for some permutation $G$-lattice 
$P$ and a $G$-lattice $M^\prime$. 
$M$ is called {\it flabby} (or {\it flasque}) if $\widehat H^{-1}(H,M)=0$ 
for any subgroup $H$ of $G$ where $\widehat H$ is the Tate cohomology. 
$M$ is called {\it coflabby} (or {\it coflasque}) if $H^1(H,M)=0$
for any subgroup $H$ of $G$. 
We say that two $G$-lattices $M_1$ and $M_2$ are {\it similar} 
if there exist permutation $G$-lattices $P_1$ and $P_2$ such that 
$M_1\oplus P_1\simeq M_2\oplus P_2$. 
The set of similarity classes becomes a commutative monoid 
with respect to the sum $[M_1]+[M_2]:=[M_1\oplus M_2]$ 
and the zero $0=[P]$ where $P$ is a permutation $G$-lattice. 
For a $G$-lattice $M$, there exists a short exact sequence of $G$-lattices 
$0 \rightarrow M \rightarrow P \rightarrow F \rightarrow 0$
where $P$ is permutation and $F$ is flabby which is called a 
{\it flabby resolution} of $M$ 
(see Endo and Miyata \cite[Lemma 1.1]{EM75}, 
Colliot-Th\'el\`ene and Sansuc \cite[Lemma 3]{CTS77}, 
Manin \cite[Appendix, page 286]{Man86}). 
The similarity class $[F]$ of $F$ is determined uniquely 
and is called {\it the flabby class} of $M$. 
We denote the flabby class $[F]$ of $M$ by $[M]^{fl}$. 
We say that $[M]^{fl}$ is invertible if $[M]^{fl}=[E]$ for some 
invertible $G$-lattice $E$. 
For $G$-lattice $M$, 
it is not difficult to see that 
\begin{align*}
\textrm{permutation}\ \ 
\Rightarrow\ \ 
&\textrm{stably\ permutation}\ \ 
\Rightarrow\ \ 
\textrm{invertible}\ \ 
\Rightarrow\ \ 
\textrm{flabby\ and\ coflabby}\\
&\hspace*{8mm}\Downarrow\hspace*{34mm} \Downarrow\\
&\hspace*{7mm}[M]^{fl}=0\hspace*{10mm}\Rightarrow\hspace*{5mm}[M]^{fl}\ 
\textrm{is\ invertible}.
\end{align*}

For algebraic $k$-tori $T$, we see that 
$[\widehat{T}]^{fl}=[{\rm Pic}\,\overline{X}]$ 
where $X$ is a smooth $k$-compactification of $T$, 
i.e. smooth projective $k$-variety $X$ containing $T$ as a dense 
open subvariety, 
and $\overline{X}=X\times_k\overline{k}$ 
(see Voskresenskii \cite[Section 4, page 1213]{Vos69}, \cite[Section 3, page 7]{Vos70}, \cite{Vos74}, \cite[Section 4.6]{Vos98}, Kunyavskii \cite[Theorem 1.9]{Kun07} and Colliot-Th\'el\`ene \cite[Theorem 5.1, page 19]{CT07} for any field $k$). 

%%%%%%%%%%
The flabby class $[M]^{fl}=[\widehat{T}]^{fl}$ 
plays crucial role in the rationality problem 
for $L(M)^G\simeq k(T)$ 
as follows (see %also %Voskresenskii's fundamental book 
Colliot-Th\'el\`ene and Sansuc \cite[Section 2]{CTS77}, 
\cite[Proposition 7.4]{CTS87}, 
Voskresenskii \cite[Section 4.6]{Vos98}, 
Kunyavskii \cite[Theorem 1.7]{Kun07}, 
Colliot-Th\'el\`ene \cite[Theorem 5.4]{CT07}, 
Hoshi and Yamasaki \cite[Section 1]{HY17}): 
%%%%%%%%%%%%%%%%%%%%%
%
\begin{theorem}%[Endo and Miyata, Voskresenskii, Saltman]
\label{th2-1}
Let $L/k$ be a finite Galois extension with Galois group $G={\rm Gal}(L/k)$ 
and $M$, $M^\prime$ be $G$-lattices.\\
{\rm (i)} $(${\rm Endo and Miyata} \cite[Theorem 1.6]{EM73}$)$ 
$[M]^{fl}=0$ if and only if $L(M)^G$ is stably $k$-rational.\\
{\rm (ii)} $(${\rm Voskresenskii} \cite[Theorem 2]{Vos74}$)$ 
$[M]^{fl}=[M^\prime]^{fl}$ if and only if $L(M)^G$ and $L(M^\prime)^G$ 
are stably $k$-isomorphic. %i.e. there exist algebraically independent 
%elements $x_1,\ldots,x_m$ over $L(M)^G$ and 
%$y_1,\ldots,y_n$ over $L(M^\prime)^G$ such that 
%$L(M)^G(x_1,\ldots,x_m)\simeq L(M^\prime)^G(y_1,\ldots,y_n)$.\\
{\rm (iii)} $(${\rm Saltman} \cite[Theorem 3.14]{Sal84}$)$ 
$[M]^{fl}$ is invertible if and only if $L(M)^G$ is 
retract $k$-rational.
\end{theorem}

%%%
Let $K/k$ be a separable field extension of degree $m$ 
and $L/k$ be the Galois closure of $K/k$. 
Let $G={\rm Gal}(L/k)$ and $H={\rm Gal}(L/K)$ with $[G:H]=m$. 
The norm one torus $R^{(1)}_{K/k}(\bG_m)$ has the 
Chevalley module $J_{G/H}$ as its character module 
and the field $L(J_{G/H})^G$ as its function field 
where $J_{G/H}=(I_{G/H})^\circ={\rm Hom}_\bZ(I_{G/H},\bZ)$ 
is the dual lattice of $I_{G/H}={\rm Ker}\ \varepsilon$ and 
$\varepsilon : \bZ[G/H]\rightarrow \bZ$ is the augmentation map 
(see \cite[Section 4.8]{Vos98}). 
We have the exact sequence $0\rightarrow \bZ\rightarrow \bZ[G/H]
\rightarrow J_{G/H}\rightarrow 0$ and ${\rm rank}_\bZ J_{G/H}=m-1$. 
Write $J_{G/H}=\oplus_{1\leq i\leq m-1}\bZ x_i$. 
Then the action of $G$ on $L(J_{G/H})=L(x_1,\ldots,x_{m-1})$ is 
of the form %nothing but 
(\ref{acts}). 

%%%%
The rationality problem for norm one tori is investigated 
by many authors, see e.g. 
%\cite{EM75}, \cite{CTS77}, \cite{Hur84}, \cite{CTS87}, 
%\cite{LeB95}, \cite{CK00}, \cite{LL00}, \cite{Flo}, \cite{End11}, 
\cite{HY17}, \cite{HHY20}, \cite{HY21} and the references therein. 
Let $C_n$ be the cyclic group of order $n$. 
\begin{theorem}[{Endo and Miyata \cite[Theorem 1.5]{EM75}, Saltman \cite[Theorem 3.14]{Sal84}}]\label{th2-2}
Let $L/k$ be a finite Galois field extension and $G={\rm Gal}(L/k)$. 
Then the following conditions are equivalent:\\
{\rm (i)} $R^{(1)}_{L/k}(\bG_m)$ is retract $k$-rational;\\
{\rm (ii)} all the Sylow subgroups of $G$ are cyclic. 
\end{theorem}

\begin{theorem}[{Endo and Miyata \cite[Theorem 2.3]{EM75}, Colliot-Th\'{e}l\`{e}ne and Sansuc \cite[Proposition 3]{CTS77}}]\label{th2-3}
Let $L/k$ be a finite Galois field extension and $G={\rm Gal}(L/k)$. 
Then the following conditions are equivalent:\\
{\rm (i)} $R^{(1)}_{L/k}(\bG_m)$ is stably $k$-rational;\\
{\rm (ii)} $G\simeq C_m$ $(m\geq 1)$ or 
$G\simeq C_m\times \langle x,y\mid x^n=y^{2^d}=1,
yxy^{-1}=x^{-1}\rangle$ $(m\geq 1: odd, n\geq 3: odd, d\geq 1)$ 
with ${\rm gcd}\{m,n\}=1$;\\
{\rm (iii)} $G\simeq \langle s,t\mid s^m=t^{2^d}=1, tst^{-1}=s^r\rangle$ 
$(m\geq 1: odd, d\geq 0)$ with $r^2\equiv 1\pmod{m}$;\\
{\rm (iv)} all the Sylow subgroups of $G$ are cyclic and $H^4(G,\bZ)\simeq \widehat H^0(G,\bZ)$ 
where $\widehat H$ is the Tate cohomology.
\end{theorem}
\begin{theorem}[Endo {\cite[Theorem 2.1]{End11}}]\label{th2-4}
Let $K/k$ be a finite non-Galois, separable field extension 
and $L/k$ be the Galois closure of $K/k$. 
Assume that the Galois group of $L/k$ is nilpotent. 
Then %the norm one torus 
$R^{(1)}_{K/k}(\bG_m)$ is not retract $k$-rational.
\end{theorem}
\begin{theorem}[Endo {\cite[Theorem 3.1]{End11}}]\label{th2-5}
Let $K/k$ be a finite non-Galois, separable field extension 
and $L/k$ be the Galois closure of $K/k$. 
Let $G={\rm Gal}(L/k)$ and $H={\rm Gal}(L/K)\leq G$. 
Assume that all the Sylow subgroups of $G$ are cyclic. 
Then %the norm one torus 
$R^{(1)}_{K/k}(\bG_m)$ is retract $k$-rational, 
and the following conditions are equivalent:\\
{\rm (i)} 
$R^{(1)}_{K/k}(\bG_m)$ is stably $k$-rational;\\
{\rm (ii)} 
$G\simeq C_m\times D_n$ $(m\geq 1: odd, n\geq 3: odd)$ 
with ${\rm gcd}\{m,n\}=1$ and $H\simeq C_2$;\\
{\rm (iii)} 
$H\simeq C_2$ and $G\simeq C_r\rtimes H$ $(r\geq 3: odd)$ 
where $H$ acts non-trivially on $C_r$. 
\end{theorem}

%%%%%%%%%%%%%%%%%%%%%%%%%%%%%%%%%%%%%%%%%%%%%%%%%%%%%%%%%%%%%%%%%
\section{Proof of Theorem {\ref{mainth}}}\label{seProof}

Let $K/k$ be a finite non-Galois separable field extension 
and $L/k$ be the Galois closure of $K/k$. 
Let $D_n$ be the dihedral group of order $2n$ $(n\geq 3)$ 
and $Z(D_n)$ be the center of $D_n$. 
Assume that $G={\rm Gal}(L/k)\simeq D_n
=\langle x,y\mid x^n=y^2=1,y^{-1}xy=x^{-1}\rangle$ $(n\geq 4:$ even$)$ 
and $H={\rm Gal}(L/K)\leq G$. 
Then it follows from the condition 
$\bigcap_{\sigma\in G} H^\sigma=\{1\}$ 
that $|H|=2$ with $H\neq Z(D_n)$. 
We prepare the following lemma: 
\begin{lemma}\label{lem3.1}
Let $K/k$ be a finite separable field extension 
and $L/k$ be the Galois closure of $K/k$. 
Let $G={\rm Gal}(L/k)$ and $H={\rm Gal}(L/K)\leq G$. 
Let ${\rm Aut}(G)$ be the group of automorphisms of $G$ 
and $K^\varphi=L^{\varphi(H)}$ $(\varphi\in{\rm Aut}(G))$. 
Then 
%{\rm (i)} $[J_{G/H}]^{fl}=0$ if and only if 
%$[J_{G/\varphi(H)}]^{fl}=0$ for any $\varphi\in {\rm Aut}(G)$;\\
%{\rm (ii)} $[J_{G/H}]^{fl}$ is invertible if and only if 
%$[J_{G/\varphi(H)}]^{fl}$ is ivertible for any $\varphi\in {\rm Aut}(G)$.
$R^{(1)}_{K/k}(\bG_m)$ is stably $k$-rational 
$($resp. retract $k$-rational$)$ 
if and only if 
$R^{(1)}_{K^\varphi/k}(\bG_m)$ is stably $k$-rational 
$($resp. retract $k$-rational$)$ for any $\varphi\in{\rm Aut}(G)$. 
\end{lemma}
\begin{proof}
We should just consider $L^{\varphi(H)}=K^\varphi=k(\theta^\prime_i)$ 
instead of $L^H=K=k(\theta_i)$. 
\end{proof}

{\it Proof of Theorem \ref{mainth}.} 
Assume that 
$G\simeq D_n=\langle x,y\mid x^n=y^2=1,y^{-1}xy=x^{-1}\rangle$ and 
$n=2m$ $(m\geq 2)$. 
The problem is determined up to conjugacy in $G$ 
because $J_{G/H}\simeq J_{G/H^\sigma}$ $(\sigma\in G)$ as $G$-lattices 
(conjugate just corresponds to base change). 
Then we see that $H$ is conjugate to 
$\langle y\rangle$ or $\langle xy\rangle$ in $G$. 
The group ${\rm Aut}(G)$ permutes 
$\langle y\rangle$ and $\langle xy\rangle$. 
Hence, by Lemma \ref{lem3.1}, we may assume that 
$H\simeq C_2=\langle xy\rangle$. 

We may assume that 
$G\simeq D_n=\langle x,y\mid x^n=y^2=1,y^{-1}xy=x^{-1}\rangle\leq S_n$ 
where $x=(1\ 2\ \cdots\ n)$, $y=(1\ n)(2\ n-1)\cdots (m\ m+1)$, 
$xy=(2\ n)(3\ n-1)\cdots (m\ m+2)$ 
with $m=n/2$.\\

{\rm (i)} Case $n\equiv 0\ ({\rm mod}\ 4)$. 
Write $n=2m$ $(m\geq 2:$ even$)$. 
We take $G=\langle x,y\mid x^n=y^2=1,y^{-1}xy=x^{-1}\rangle\simeq D_n$ and 
$H=\langle xy\rangle\simeq C_2$. 
We consider $[J_{G/H}]^{fl}|_{G^\prime}$ which is obtained by 
the restriction of the action of $G$ on $[J_{G/H}]^{fl}$ to the subgroup 
$G^\prime=\langle x^2,y\rangle\simeq D_m$ with $H^\prime=H\cap G^\prime=1$. 
Then $G^\prime$ becomes a transitive subgroup of $S_n$ and 
$[J_{G/H}]^{fl}|_{G^\prime}=[J_{G^\prime}]^{fl}$ is not invertible by 
Theorem \ref{th2-1} and Theorem \ref{th2-2}. 
This implies that $[J_{G/H}]^{fl}$ is not invertible 
(see \cite[Remarque R2]{CTS77}, \cite[Lemma 2.17]{HY17}) 
and hence $R^{(1)}_{K/k}(\bG_m)$ is not retract $k$-rational 
(by Theorem \ref{th2-1}).\\
%%%%%%%%%%%%%%%%%%%%%%%%%%%%%%%%%%%%%%%%%%%%%%%%%%%%%%%%%%%%%%%%%

{\rm (ii)} Case $n\equiv 2\ ({\rm mod}\ 4)$. 
Write $n=2m$ $(m\geq 3:$ odd$)$. 
We take $G=\langle x,y\mid x^n=y^2=1,y^{-1}xy=x^{-1}\rangle\simeq D_n$ and 
$H=\langle xy\rangle\simeq C_2$. 
We have the exact sequence 
\begin{align*}
0\to I_{G/H}\to \bZ[G/H]\xrightarrow{\varepsilon} \bZ\to 0
\end{align*}
where $\varepsilon$ is the augmentation map. 
We will construct an exact sequence of $G$-lattices 
\begin{align*}
0\to C\to P\to I_{G/H}\to 0
\end{align*}
with $P$ permutation and $C$ stably permutation, i.e. $[C]=0$. 
In particular, $C$ and the dual $(C)^\circ={\rm Hom}_\bZ(C,\bZ)$ 
are invertible. 
Then we get a flabby resolution 
\begin{align*}
0\to J_{G/H}\to (P)^\circ\to (C)^\circ\to 0
\end{align*}
of $J_{G/H}$. 
This implies that $[J_{G/H}]^{fl}=[(C)^\circ]=0$ and hence 
$R^{(1)}_{K/k}(\bG_m)$ is stably $k$-rational (by Theorem \ref{th2-1}).\\

Let $e_1,\ldots,e_n$ be the standard $\bZ$-basis of $\bZ[G/H]$ 
via the left regular representation of $G/H$: 
\begin{align*}
x&:e_1\mapsto\cdots\mapsto e_n\mapsto e_1,\\
y&:e_i\leftrightarrow e_{n+1-i}\ (1\leq i\leq m).
\end{align*}
Define $f_i:=e_i-e_{i+1}$ $(1\leq i\leq n-1)$ and $f_n:=e_n-e_1$. 
Then $f_1,\ldots,f_{n-1}$ becomes a $\bZ$-basis of $I_{G/H}$ 
with $\sum_{i=1}^n f_i=0$, i.e. $f_n=-\sum_{i=1}^{n-1}f_i$: 
\begin{align*}
x&:f_1\mapsto\cdots\mapsto f_{n-1}\mapsto -\sum_{i=1}^{n-1}f_i,\\
y&:f_i\leftrightarrow -f_{n-i}\ (1\leq i\leq m).
\end{align*}
We consider a permutation $G$-lattice 
$P\simeq\bZ[G/Z(G)]\oplus\bZ[G/H]\simeq \bZ[G/\langle x^m\rangle]\oplus\bZ[G/\langle xy\rangle]\simeq \bZ[D_m]\oplus\bZ[D_n/C_2]$ 
$(n=2m)$ where $H={\rm Stab}_{1}(G)$ 
with $\bZ$-basis 
$\alpha_1,\ldots,\alpha_m$, 
$\beta_1,\ldots,\beta_m$, 
$\gamma_1,\ldots,\gamma_n$ on which $G$ acts by 
\begin{align*}
x&:\alpha_1\mapsto\cdots\mapsto\alpha_m\mapsto\alpha_1, 
\beta_1\mapsto\cdots\mapsto\beta_m\mapsto\beta_1, 
\gamma_1\mapsto\cdots\mapsto\gamma_n\mapsto\gamma_1,\\
y&:\alpha_i\leftrightarrow\beta_{m-i}\ (1\leq i\leq m-1), 
\alpha_m\leftrightarrow\beta_m, 
\gamma_j\leftrightarrow\gamma_{n+1-j}\ (1\leq j\leq n). 
\end{align*}
Note that ${\rm rank}_\bZ P=2m+n=2n$. 

We define a $G$-homomorphism 
\begin{align*}
\varphi: P\rightarrow I_{G/H},\ 
&\alpha_i\mapsto f_i+f_{m+i},\ 
\beta_i\mapsto -(f_i+f_{m+i})\ (1\leq i\leq m),\\ 
&\gamma_j\mapsto x^{j-1}\left(\sum_{l=1}^{m-1}f_l+f_{m+1}\right)\ 
(1\leq j\leq n). 
\end{align*}
We claim that $\varphi$ is surjective. 
Indeed, we see that 
$\sum_{l=2}^{m-1}f_l=\varphi(\gamma_1)-\varphi(\alpha_1)$,  
$f_m+f_{m+2}=-\varphi(\gamma_1)+\varphi(\alpha_1)+\varphi(\gamma_2)$, 
$f_i+f_{i+2}=x^{i-m}(f_m+f_{m+2})\in {\rm Image}(\varphi)$ 
and $\sum_{l=i}^{i+3}f_l\in {\rm Image}(\varphi)$. 
Thus it follows from 
$\sum_{l=2}^{m-1}f_l, \sum_{l=i}^{i+3}f_l\in {\rm Image}(\varphi)$ 
that $f_1$ (resp. $f_2$) $\in {\rm Image}(\varphi)$ 
when $m\equiv 1$ $({\rm mod}\ 4)$ 
(resp. $m\equiv 3$ $({\rm mod}\ 4)$) 
and hence $f_i=x^{i-1}(f_1)$ (resp. $f_i=x^{i-2}(f_2)$) 
$\in {\rm Image}(\varphi)$. 
This implies that $\varphi$ is surjective. 

We get an exact sequence of $G$-lattices 
\begin{align*}
0\to C\to P\xrightarrow{\varphi} I_{G/H}\to 0
\end{align*}
where $C={\rm Ker}(\varphi)$ with ${\rm rank}_\bZ C=n+1=2m+1$. 
We find a $\bZ$-basis 
$a_i=\alpha_i+\beta_i$, 
$b_i=x^{i-1}(\alpha_1+\beta_m-\gamma_1-\gamma_{m+1})$ 
$(1\leq i\leq m)$, $c_1=\sum_{i=1}^m\alpha_i$ 
of $C$ 
and the action of $G$ on 
$C=\langle a_1,\ldots,a_m,b_1,\ldots,b_m,c_1\rangle_\bZ$ is given by 
\begin{align*}
x&:a_1\mapsto\cdots\mapsto a_m\mapsto a_1, 
b_1\mapsto\cdots\mapsto b_m\mapsto b_1, 
c_1\mapsto c_1,\\
y&:a_i\leftrightarrow a_{m-i}, 
b_i\leftrightarrow b_{m+1-i}\ (1\leq i\leq \tfrac{m-1}{2}), 
a_m\mapsto a_m, 
b_{\frac{m+1}{2}}\mapsto b_{\frac{m+1}{2}}, 
c_1\mapsto \sum_{i=1}^m a_i-c_1. 
\end{align*}
We remark that the action of $G$ on $C$ is not faithful. 
Indeed, the subgroup $Z(G)=\langle x^m\rangle\simeq C_2$ acts on $C$ trivially 
and hence we can regard $C$ as a $D_m$-lattice where $D_m=G/Z(G)=G/\langle x^m\rangle$ $(n=2m)$. 

Take an element $c_2=\sum_{i=1}^m\beta_i\in P$. 
Then we have 
\begin{align*}
x&: c_1\mapsto c_1, c_2\mapsto c_2,\\
y&: c_1\mapsto c_2, c_2\mapsto c_1. 
\end{align*}
Then we consider the trivial $G$-lattice 
$\langle z_0\rangle_\bZ\simeq\bZ$ and extend the map $\varphi$ from $P$ 
to $P\oplus \langle z_0\rangle_\bZ$ by $\widetilde{\varphi}(z_0)=0$: 
\begin{align*}
0\to C\oplus\langle z_0\rangle_\bZ\to P\oplus\langle z_0\rangle_\bZ
\xrightarrow{\widetilde{\varphi}} I_{G/H}\to 0
\end{align*}
where $C\oplus\langle z_0\rangle_\bZ={\rm Ker}(\widetilde{\varphi})$ 
with ${\rm rank}_\bZ (C\oplus\langle z_0\rangle_\bZ)=(n+1)+1=2m+2$. 

Because ${\rm gcd}\{2,m\}=1$, there exist $u,v\in\bZ$ such that 
$2u+mv=1$. 
Then we can get a $\bZ$-basis 
$a_i^\prime:=a_i+vz_0$, 
$b_i$ $(1\leq i\leq m)$, 
$c_l^\prime:=c_l-uz_0$ $(1\leq l\leq 2)$ 
of $C\oplus\langle z_0\rangle_\bZ\simeq 
\bZ[D_m/\langle x^2y\rangle]\oplus \bZ[D_m/\langle xy\rangle]\oplus 
\bZ[D_m/\langle x\rangle]$ 
with $\bZ$-rank $(2m+1)+1=m+m+2$ 
on which $D_m= G/\langle x^m\rangle$ acts by 
\begin{align*}
x&:a_1^\prime\mapsto\cdots\mapsto a_m^\prime\mapsto a_1^\prime, 
b_1\mapsto\cdots\mapsto b_m\mapsto b_1, 
c_1^\prime\mapsto c_1^\prime, c_2^\prime\mapsto c_2^\prime,\\
y&:a_i^\prime\leftrightarrow a_{m-i}^\prime, 
b_i\leftrightarrow b_{m+1-i}\ (1\leq i\leq \tfrac{m-1}{2}), 
a_m^\prime\mapsto a_m^\prime, 
b_{\frac{m+1}{2}}\mapsto b_{\frac{m+1}{2}}, 
c_1^\prime\leftrightarrow c_2^\prime.
\end{align*}
Indeed, we can confirm that 
\begin{align*}
&\ \ \ \begin{array}{ccc|cc}
\multicolumn{3}{c}{\overbrace{\hspace{15mm}}^{m}} & 
\multicolumn{2}{c}{\,\overbrace{\hspace{12mm}}^{2}}
\end{array}\\
\det&\left(
\begin{array}{ccc|cc}
 1 &  &  & 0 & v \\
  & \ddots &  & \vdots & \vdots \\
  & & 1 & 0 & v \\\hline
 0 & \cdots & 0 & 1 & -u \\
 1 & \cdots & 1 & -1 & -u \\
\end{array}
\right)=-2u-mv=-1.
\end{align*}
Then we find that $C\oplus\langle z_0\rangle_\bZ\simeq 
\bZ[D_m/\langle x^2y\rangle]\oplus \bZ[D_m/\langle xy\rangle]\oplus 
\bZ[D_m/\langle x\rangle]$ 
is a permutation $G$-lattice ($D_m$-lattice) and hence $[C]=0$. 

The last statement follows from the exact sequence 
\begin{align*}
0\to J_{G/H}\to (P)^\circ\oplus\bZ\to (C)^\circ\oplus\bZ\to 0
\end{align*}
with $(P)^\circ\oplus\bZ\simeq 
\bZ[G/\langle x^m\rangle]\oplus\bZ[G/\langle xy\rangle]\oplus\bZ$, 
$(C^\circ)\oplus\bZ\simeq 
\bZ[D_m/\langle x^2y\rangle]\oplus \bZ[D_m/\langle xy\rangle]\oplus 
\bZ[D_m/\langle x\rangle]$ 
because 
it follows from 
Endo and Miyata %\cite[Proposition 1.1, Corollary 1.3]{EM73} 
\cite[Proposition 1.10]{EM73} and Lenstra \cite[Proposition 1.5]{Len74} 
(see also Ono \cite[Proposition 1.2.2]{Ono63}, 
Swan \cite[Lemma 3.1]{Swa10}, Hoshi, Kang and Kitayama \cite[Proof of Theorem 6.5]{HKK14}) that 
$L(J_{G/H}\oplus (C)^\circ\oplus\bZ)^G\simeq 
L((P)^\circ\oplus\bZ)^G$.\qed\\

As a consequence of the proof of Theorem \ref{mainth} above, we get: 
\begin{corollary}\label{cor3.2}
Let $C={\rm Ker(\varphi)}$ be a $D_m$-lattice %$($resp. $D_m$-lattice$)$ 
with ${\rm rank}_\bZ C=2m+1$ $(m\geq 3: odd)$ %$=2m+1$ 
given as in {\rm (ii)} Case $n\equiv 2\ ({\rm mod}\ 4)$ in 
the proof of Theorem \ref{mainth}. %$(D_m=D_n/Z(D_n)$ and $m\geq 3$: odd$)$. 
Then we have $\widehat{H}^0(D_m,C)$ %$\simeq \widehat{H}^0(D_m,C)
$\simeq\bZ/2\bZ$ 
where $\widehat{H}^0(G,M)=M^G/\{\sum_{g\in G} g\cdot m\mid m\in M\}$ 
is the 0th Tate cohomology. 
In particular, the $D_m$-lattice $C$ is not permutation but stably permutation which satisfies $C\oplus\bZ\simeq 
\bZ[D_m/\langle x^2y\rangle]\oplus \bZ[D_m/\langle xy\rangle]\oplus 
\bZ[D_m/\langle x\rangle]$ with $\bZ$-rank $(2m+1)+1=m+m+2$.  
\end{corollary}
\begin{proof}
Recall that we can regard $D_n$-lattice $C$ as $D_m$-lattice via 
$D_m=D_n/Z(D_n)=D_n/\langle x^m\rangle$. 
By Shapiro's lemma, we have $\widehat{H}^0(D_m,\bZ[D_m/H])
\simeq \widehat{H}^0(H,\bZ)\simeq \bZ/|H|\bZ$ for any subgroup $H\leq D_m$.  
Hence it follows from the isomorphism $C\oplus\bZ\simeq 
\bZ[D_m/\langle x^2y\rangle]\oplus \bZ[D_m/\langle xy\rangle]\oplus 
\bZ[D_m/\langle x\rangle]$ given in the proof of Theorem \ref{mainth} 
that $\widehat{H}^0(D_m,C\oplus\bZ)\simeq \bZ/2\bZ\oplus\bZ/2\bZ\oplus\bZ/m\bZ$.
Because $\widehat{H}^0(D_m,\bZ)\simeq \bZ/2m\bZ\simeq\bZ/2\bZ\oplus\bZ/m\bZ$ 
($m:$ odd), we get that $\widehat{H}^0(D_m,C)\simeq \bZ/2\bZ$. 
Suppose that $C=\oplus_{i=1}^r\bZ[D_m/H_i]$ is permutation. 
Then $\widehat{H}^0(D_m,C)\simeq\oplus_{i=1}^r\bZ/|H_i|\bZ$. 
This implies that the $\bZ$-rank of $C$ is a multiple of $m$ because $|H_i|=1$ or $2$. 
Contradiction. 
\end{proof}

%%%%%%%%%%%%%%%%%%%%%%%%%%%%%%%%%%%%%%%%%%%%%%%%%%%%%%%%%%%%%%%%%
\section{Appendix: refinements of the proof of stably rational cases with $G\simeq D_n$ $(n\geq 3: odd)$}\label{seAlt}

We give refinements of the proof %alternative proofs 
of stably rational cases with 
$G\simeq D_n$ $(n\geq 3:$ odd$)$ of Endo and Miyata's theorems 
\cite[Theorem 2.3]{EM75} with $H=\{1\}$ and 
\cite[Theorem 3.1]{End11} with $H\simeq C_2$, 
using a technique similar to that given in the proof 
of Theorem \ref{mainth}. 
We remark that the original proofs given as in 
$(1)\Rightarrow (2)$ of \cite[Theorem 2.3]{EM75} 
and $(1)\Rightarrow (2)$ of \cite[Theorem 3.1]{End11} 
use an induction argument on the number of prime divisors of $n$ 
and our refinements %of the proof alternative ones 
can be obtained by constructing explicit flabby resolutions 
of $J_{G/H}$ for $H=\{1\}$ and $H\simeq C_2$ respectively.\\

{\it Refinement of the proof of Theorem \ref{th2-3} {\rm (Endo and Miyata \cite[Theorem 2.3]{EM75})}: stably rational cases with $G\simeq D_n$ $(n\geq 3:$ odd$)$ and $H=\{1\}$.} 

Let $L/k$ be a finite Galois field extension 
and $G={\rm Gal}(L/k)$. 
We consider the case where 
$G\simeq D_n=\langle x,y\mid x^n=y^2=1,y^{-1}xy=x^{-1}\rangle\leq S_{2n}$ 
$(n\geq 3:$ odd$)$ and $H=\{1\}$ 
where $x=(1\ 3\ \cdots\ 2n-1)(2\ 4\ \cdots\ 2n)$, 
$y=(1\ 2n)(2\ 2n-1)\cdots (n\ n+1)$. 

Take the exact sequence 
\begin{align*}
0\to I_G\to \bZ[G]\xrightarrow{\varepsilon} \bZ\to 0 
\end{align*}
where $\varepsilon$ is the augmentation map. 
We will construct an exact sequence of $G$-lattices 
\begin{align*}
0\to C\to P\to I_G\to 0
\end{align*}
with $P$ permutation and $C$ stably permutation, i.e. $[C]=0$. 
Then we get a flabby resolution 
\begin{align*}
0\to J_G\to (P)^\circ\to (C)^\circ\to 0
\end{align*}
of $J_G$. 
This implies that $[J_G]^{fl}=[(C)^\circ]=0$ and hence 
$R^{(1)}_{K/k}(\bG_m)$ is stably $k$-rational (by Theorem \ref{th2-1}).\\

%%%

Let $e_1,\ldots,e_{2n}$ be the standard $\bZ$-basis of $\bZ[G]$: 
\begin{align*}
x&:e_1\mapsto e_3\mapsto \cdots\mapsto e_{2n-1}\mapsto e_1, 
e_2\mapsto e_4\mapsto\cdots\mapsto e_{2n},\\
y&:e_i\leftrightarrow e_{2n+1-i}\ (1\leq i\leq n).
\end{align*}
Define $f_i:=e_i-e_{i+1}$ $(1\leq i\leq 2n-1)$ and $f_{2n}:=e_{2n}-e_1$. 
Then $f_1,\ldots,f_{2n-1}$ becomes a $\bZ$-basis of $I_G$ 
with $\sum_{i=1}^{2n} f_i=0$, i.e. $f_{2n}=-\sum_{i=1}^{2n-1}f_i$: 
\begin{align*}
x&:f_1\mapsto\cdots\mapsto f_{2n-1}\mapsto -\sum_{i=1}^{2n-1}f_i,\\
y&:f_i\leftrightarrow -f_{2n-i}\ (1\leq i\leq n).
\end{align*}
We consider a permutation $G$-lattice 
$P\simeq\bZ[G]\oplus\bZ[G]$ with $\bZ$-basis 
$\alpha_1,\alpha_3,\ldots,\alpha_{2n-1}$, 
$\beta_1,\beta_3,\ldots,\beta_{2n-1}$, 
$\alpha_2,\alpha_4,\ldots,\alpha_{2n}$, 
$\beta_2,\beta_4,\ldots,\beta_{2n}$ on which $G$ acts by 
\begin{align*}
x&:\alpha_1\mapsto \alpha_3\mapsto \cdots\mapsto\alpha_{2n-1}\mapsto\alpha_1, 
\beta_1\mapsto \beta_3\mapsto \cdots\mapsto\beta_{2n-1}\mapsto\beta_1,\\
&\ \ \ \alpha_2\mapsto \alpha_4\mapsto \cdots\mapsto\alpha_{2n}\mapsto\alpha_2, 
\beta_2\mapsto \beta_4\mapsto \cdots\mapsto\beta_{2n}\mapsto\beta_2,\\
y&:\alpha_i\leftrightarrow\beta_{2n-i}\ (1\leq i\leq 2n-1), 
\alpha_{2n}\leftrightarrow\beta_{2n}. 
\end{align*}
%Note that ${\rm rank}_\bZ P=4n$. 

We define a surjective $G$-homomorphism 
\begin{align*}
\varphi: P\rightarrow I_G,\ 
&\alpha_i\mapsto f_i,\ 
\beta_i\mapsto -f_i\ (1\leq i\leq 2n). 
\end{align*}
Then we get an exact sequence of $G$-lattices 
\begin{align*}
0\to C\to P\xrightarrow{\varphi} I_G\to 0
\end{align*}
where $C={\rm Ker}(\varphi)$ with ${\rm rank}_\bZ C=2n+1$. 
We find a $\bZ$-basis 
$a_i=\alpha_i+\beta_i$\ $(1\leq i\leq n)$, 
$b_1=\sum_{i=1}^{2n}\alpha_i$ of $C$ 
and the action of $G$ on 
$C=\langle a_1,\ldots,a_{2n},b_1\rangle_\bZ$ is given by 
\begin{align*}
x&:a_1\mapsto a_3\mapsto \cdots\mapsto a_{2n-1}\mapsto a_1, 
a_2\mapsto a_4\mapsto \cdots\mapsto a_{2n}\mapsto a_2, 
b_1\mapsto b_1,\\
y&:a_i\leftrightarrow a_{2n-i}\ (1\leq i\leq n-1), 
a_{2n}\mapsto a_{2n}, 
b_1\mapsto \sum_{i=1}^{2n} a_i-b_1. 
\end{align*}

As in the proof of Theorem \ref{mainth}, 
we consider the trivial $G$-lattice 
$\langle z_0\rangle_\bZ\simeq\bZ$ and extend the map $\varphi$ from $P$ 
to $P\oplus \langle z_0\rangle_\bZ$ by $\widetilde{\varphi}(z_0)=0$: 
\begin{align*}
0\to C\oplus\langle z_0\rangle_\bZ\to P\oplus\langle z_0\rangle_\bZ
\xrightarrow{\widetilde{\varphi}} I_G\to 0
\end{align*}
where $C\oplus\langle z_0\rangle_\bZ={\rm Ker}(\widetilde{\varphi})$ 
with ${\rm rank}_\bZ (C\oplus\langle z_0\rangle_\bZ)=(2n+1)+1$. 

Because ${\rm gcd}\{2,n\}=1$, there exist $u,v\in\bZ$ such that 
$2u+nv=1$. 
Then we can get a $\bZ$-basis 
$a_{2i-1}^\prime:=a_{2i-1}+vz_0$, $a_{2i}$ $(1\leq i\leq n)$, 
$b_l^\prime:=b_l-uz_0$ $(1\leq l\leq 2)$ 
of $C\oplus\langle z_0\rangle_\bZ\simeq 
\bZ[G/\langle xy\rangle]\oplus \bZ[G/\langle x^2y\rangle]\oplus 
\bZ[G/\langle x\rangle]$ 
with $\bZ$-rank $(2n+1)+1=n+n+2$ 
on which $G$ acts by 
\begin{align*}
x&:a_1^\prime\mapsto a_3^\prime\mapsto \cdots\mapsto a_{2n-1}^\prime\mapsto a_1^\prime, 
a_2\mapsto a_4\mapsto \cdots\mapsto a_{2n}\mapsto a_2, 
b_1^\prime\mapsto b_1^\prime, b_2^\prime\mapsto b_2^\prime,\\
y&:a_{2i-1}^\prime\leftrightarrow a_{2n-(2i-1)}^\prime\ (1\leq i\leq n), 
a_{2j}\leftrightarrow a_{2n-2j}\ (1\leq j\leq n-1), 
a_{2n}\mapsto a_{2n}, 
b_1^\prime\leftrightarrow b_2^\prime.
\end{align*}
Indeed, we can confirm that 
\begin{align*}
&\ \ \ \begin{array}{ccc|cc}
\multicolumn{3}{c}{\overbrace{\hspace{15mm}}^{n}} & 
\multicolumn{2}{c}{\,\overbrace{\hspace{12mm}}^{2}}
\end{array}\\
\det&\left(
\begin{array}{ccc|cc}
 1 &  &  & 0 & v \\
  & \ddots &  & \vdots & \vdots \\
  & & 1 & 0 & v \\\hline
 0 & \cdots & 0 & 1 & -u \\
 1 & \cdots & 1 & -1 & -u \\
\end{array}
\right)=-2u-nv=-1.
\end{align*}
Then $C\oplus\langle z_0\rangle_\bZ\simeq\bZ[G/\langle xy\rangle]\oplus \bZ[G/\langle x^2y\rangle]\oplus 
\bZ[G/\langle x\rangle]$ 
is a permutation $G$-lattice and hence $[C]=0$.\qed\\

It follows from the exact sequence 
\begin{align*}
0\to J_{G}\to (P)^\circ\oplus\bZ\to (C)^\circ\oplus\bZ\to 0, 
\end{align*}
Endo and Miyata %\cite[Proposition 1.1, Corollary 1.3]{EM73} 
\cite[Proposition 1.10]{EM73} and Lenstra \cite[Proposition 1.5]{Len74} 
(see also Ono \cite[Proposition 1.2.2]{Ono63}, 
Swan \cite[Lemma 3.1]{Swa10}, Hoshi, Kang and Kitayama \cite[Proof of Theorem 6.5]{HKK14}) that 
$L(J_{G}\oplus (C)^\circ\oplus\bZ)^G\simeq L((P)^\circ\oplus\bZ)^G$. 
Hence we get: 
\begin{corollary}
Let $G\simeq D_n=\langle x,y\mid x^n=y^2=1,y^{-1}xy=x^{-1}\rangle\leq S_{2n}$ 
$(n\geq 3:$ odd$)$ and $H=\{1\}$. %as above. 
Then 
$R^{(1)}_{K/k}(\bG_m)\times R_{k_1/k}(\bG_{m,k_1})\times R_{k_2/k}(\bG_{m,k_2})\times R_{k_3/k}(\bG_{m,k_3})$ of dimension $(2n-1)+n+n+2=4n+1$ and 
$R_{L/k}(\bG_{m,L})\times R_{L/k}(\bG_{m,L})\times \bG_{m,k}$ 
of dimension $2n+2n+1=4n+1$ 
are birationally $k$-equivalent where 
$R_{k_i/k}(\bG_{m,k_i})$ and 
$R_{L/k}(\bG_{m,L})$ are $k$-rational with 
$k_1=L^{\langle xy\rangle}$, 
$k_2=L^{\langle x^2y\rangle}$, 
$k_3=L^{\langle x\rangle}$, 
$\langle xy\rangle\simeq 
\langle x^2y\rangle \simeq C_2$, 
$\langle x\rangle\simeq C_n$ and 
$[k_1:k]=[k_2:k]=n$, $[k_3:k]=2$. 
\end{corollary}
\begin{remark}
We see that the $G$-lattice $C$ is not permutation but stably permutation as in Corollary  \ref{cor3.2}.\\
\end{remark}
%%%%%%%%%%%%%%%%%%%%%%%%%%%%%%%%%%%%%%%%%%%%%%%%%%%%%%%%%%%%%%%%%%%%%%%%%%%%%

{\it Refinement of the proof of Theorem \ref{th2-5} {\rm (Endo \cite[Theorem 3.1]{End11})}: stably rational cases with $G\simeq D_n$ $(n\geq 3:$ odd$)$ and $H\simeq C_2$.} 

Let $K/k$ be a finite non-Galois separable field extension, 
$L/k$ be the Galois closure of $K/k$ and $G={\rm Gal}(L/k)$. 
By Lemma \ref{lem3.1}, we may assume that 
$G\simeq D_n=\langle x,y\mid x^n=y^2=1,y^{-1}xy=x^{-1}\rangle\leq S_n$ 
$(n\geq 3:$ odd$)$ and $H=\langle xy\rangle$ 
where $x=(1\ 2\ \cdots\ n)$, 
$y=(1\ n)(2\ n-1)\cdots (\tfrac{n-1}{2}\ \tfrac{n+3}{2})$, 
$xy=(2\ n)(3\ n-1)\cdots (\tfrac{n+1}{2}\ \tfrac{n+1}{2}+1)$. 

We take the exact sequence 
\begin{align*}
0\to I_{G/H}\to \bZ[G/H]\xrightarrow{\varepsilon} \bZ\to 0
\end{align*}
where $\varepsilon$ is the augmentation map. 
We will construct an exact sequence of $G$-lattices 
\begin{align*}
0\to C\to P\to I_{G/H}\to 0
\end{align*}
with $P$ permutation and $C$ stably permutation, i.e. $[C]=0$. 
Then we get a flabby resolution 
\begin{align*}
0\to J_{G/H}\to (P)^\circ\to (C)^\circ\to 0
\end{align*}
of $J_{G/H}$. 
This implies that $[J_{G/H}]^{fl}=[(C)^\circ]=0$ and hence 
$R^{(1)}_{K/k}(\bG_m)$ is stably $k$-rational (by Theorem \ref{th2-1}).\\

Let $e_1,\ldots,e_n$ be the standard $\bZ$-basis of $\bZ[G/H]$ 
via the left regular representation of $G/H$: 
\begin{align*}
x&:e_1\mapsto\cdots\mapsto e_n\mapsto e_1,\\
y&:e_i\leftrightarrow e_{n+1-i}\ (1\leq i\leq \tfrac{n-1}{2}), 
e_{\tfrac{n+1}{2}}\mapsto e_{\tfrac{n+1}{2}}.
\end{align*}
Define $f_i:=e_i-e_{i+1}$ $(1\leq i\leq n-1)$ and $f_n:=e_n-e_1$. 
Then $f_1,\ldots,f_{n-1}$ becomes a $\bZ$-basis of $I_{G/H}$ 
with $\sum_{i=1}^n f_i=0$, i.e. $f_n=-\sum_{i=1}^{n-1}f_i$: 
\begin{align*}
x&:f_1\mapsto\cdots\mapsto f_{n-1}\mapsto -\sum_{i=1}^{n-1}f_i,\\
y&:f_i\leftrightarrow -f_{n+1-i}\ (1\leq i\leq \tfrac{n+1}{2}).
\end{align*}
We consider a permutation $G$-lattice $P\simeq\bZ[G]$ 
with $\bZ$-basis 
$\alpha_1,\ldots,\alpha_n$, $\beta_1,\ldots,\beta_n$ on which $G$ acts by 
\begin{align*}
x&:\alpha_1\mapsto\cdots\mapsto\alpha_n\mapsto\alpha_1, 
\beta_1\mapsto\cdots\mapsto\beta_n\mapsto\beta_1,\\
y&:\alpha_i\leftrightarrow\beta_{n+1-i}\ (1\leq i\leq n).
\end{align*}

We define a surjective $G$-homomorphism 
\begin{align*}
\varphi: P\rightarrow I_{G/H},\ 
&\alpha_i\mapsto f_i,\ 
\beta_i\mapsto -f_i\ (1\leq i\leq n). 
\end{align*}
Then we get an exact sequence of $G$-lattices 
\begin{align*}
0\to C\to P\xrightarrow{\varphi} I_{G/H}\to 0
\end{align*}
where $C={\rm Ker}(\varphi)$ with ${\rm rank}_\bZ C=n+1$. 
We find a $\bZ$-basis 
$a_i=\alpha_i+\beta_i$ $(1\leq i\leq n)$, 
$b_1=\sum_{i=1}^n\alpha_i$ 
of $C$ 
and the action of $G$ on 
$C=\langle a_1,\ldots,a_n,b_1\rangle_\bZ$ is given by 
\begin{align*}
x&:a_1\mapsto\cdots\mapsto a_n\mapsto a_1, b_1\mapsto b_1,\\
y&:a_i\leftrightarrow a_{n+1-i}\ (1\leq i\leq \tfrac{n-1}{2}), 
a_{\tfrac{n+1}{2}}\mapsto a_{\tfrac{n+1}{2}}, 
b_1\mapsto \sum_{i=1}^n a_i-b_1. 
\end{align*}

As in the proof of Theorem \ref{mainth}, 
we consider the trivial $G$-lattice 
$\langle z_0\rangle_\bZ\simeq\bZ$ and extend the map $\varphi$ from $P$ 
to $P\oplus \langle z_0\rangle_\bZ$ by $\widetilde{\varphi}(z_0)=0$: 
\begin{align*}
0\to C\oplus\langle z_0\rangle_\bZ\to P\oplus\langle z_0\rangle_\bZ
\xrightarrow{\widetilde{\varphi}} I_{G/H}\to 0
\end{align*}
where $C\oplus\langle z_0\rangle_\bZ={\rm Ker}(\widetilde{\varphi})$ 
with ${\rm rank}_\bZ (C\oplus\langle z_0\rangle_\bZ)=(n+1)+1$. 

Because ${\rm gcd}\{2,n\}=1$, there exist $u,v\in\bZ$ such that 
$2u+nv=1$. 
Then we can get a $\bZ$-basis 
$a_i^\prime:=a_i+vz_0$ $(1\leq i\leq n)$, 
$b_l^\prime:=b_l-uz_0$ $(1\leq l\leq 2)$ 
of $C\oplus\langle z_0\rangle_\bZ\simeq 
\bZ[G/\langle xy\rangle]\oplus 
\bZ[G/\langle x\rangle]$ 
with $\bZ$-rank $(n+1)+1=n+2$
on which $G$ acts by 
\begin{align*}
x&:a_1^\prime\mapsto\cdots\mapsto a_n^\prime\mapsto a_1^\prime, 
b_1^\prime\mapsto b_1^\prime, b_2^\prime\mapsto b_2^\prime,\\
y&:a_i^\prime\leftrightarrow a_{n+1-i}^\prime\ (1\leq i\leq \tfrac{n-1}{2}), 
a_{\tfrac{n+1}{2}}^\prime\mapsto a_{\tfrac{n+1}{2}}^\prime, 
b_1^\prime\leftrightarrow b_2^\prime.
\end{align*}
Indeed, we can confirm that 
\begin{align*}
&\ \ \ \begin{array}{ccc|cc}
\multicolumn{3}{c}{\overbrace{\hspace{15mm}}^{n}} & 
\multicolumn{2}{c}{\,\overbrace{\hspace{12mm}}^{2}}
\end{array}\\
\det&\left(
\begin{array}{ccc|cc}
 1 &  &  & 0 & v \\
  & \ddots &  & \vdots & \vdots \\
  & & 1 & 0 & v \\\hline
 0 & \cdots & 0 & 1 & -u \\
 1 & \cdots & 1 & -1 & -u \\
\end{array}
\right)=-2u-nv=-1.
\end{align*}
Then we find that 
$C\oplus\langle z_0\rangle_\bZ\simeq 
\bZ[G/\langle xy\rangle]\oplus 
\bZ[G/\langle x\rangle]$ is a permutation $G$-lattice 
and hence $[C]=0$.\qed\\

It follows from the exact sequence 
\begin{align*}
0\to J_{G/H}\to (P)^\circ\oplus\bZ\to (C)^\circ\oplus\bZ\to 0, 
\end{align*}
Endo and Miyata %\cite[Proposition 1.1, Corollary 1.3]{EM73} 
\cite[Proposition 1.10]{EM73} and Lenstra \cite[Proposition 1.5]{Len74} 
(see also Ono \cite[Proposition 1.2.2]{Ono63}, 
Swan \cite[Lemma 3.1]{Swa10}, Hoshi, Kang and Kitayama \cite[Proof of Theorem 6.5]{HKK14}) that 
$L(J_{G/H}\oplus (C)^\circ\oplus\bZ)^G\simeq L((P)^\circ\oplus\bZ)^G$. 
Hence we get: 
\begin{corollary}
Let $G\simeq D_n=\langle x,y\mid x^n=y^2=1,y^{-1}xy=x^{-1}\rangle\leq S_n$ 
$(n\geq 3:$ odd$)$ and $H=\langle xy\rangle\simeq C_2$. %as above. 
Then 
$R^{(1)}_{K/k}(\bG_m)\times R_{k_1/k}(\bG_{m,k_1})\times 
R_{k_2/k}(\bG_{m,k_2})$ of dimension $(n-1)+n+2=2n+1$ and 
$R_{L/k}(\bG_{m,L})\times \bG_{m,k}$ of dimension $2n+1$ 
are birationally $k$-equivalent where 
$R_{k_i/k}(\bG_{m,k_i})$ and 
$R_{L/k}(\bG_{m,L})$ are $k$-rational with 
$k_1=L^{\langle xy\rangle}$, 
$k_2=L^{\langle x\rangle}$, 
$\langle xy\rangle \simeq C_2$, 
$\langle x\rangle\simeq C_n$ and 
$[k_1:k]=n$, $[k_2:k]=2$. 
\end{corollary}
\begin{remark}
We see that the $G$-lattice $C$ is not permutation but stably permutation as in Corollary  \ref{cor3.2}. 
\end{remark}


\begin{thebibliography}{BCTSSD85}
\bibitem[AM72]{AM72} 
M. Artin, D. Mumford, 
{\it Some elementary examples of unirational varieties which are not 
rational}, 
Proc. London Math. Soc. (3) \textbf{25} (1972) 75--95. 
\bibitem[Bea77a]{Bea77a} 
A. Beauville, {\it Prym varieties and the Schottky problem}, 
Invent. Math. \textbf{41} (1977) 149--196. 
\bibitem[Bea77b]{Bea77b} 
A. Beauville, {\it Vari\'et\'es de Prym et jacobiennes interm\'ediaires}, 
(French) 
Ann. Sci. \'Ecole Norm. Sup. (4) \textbf{10} (1977) 309--391. 
\bibitem[Bea16]{Bea16} 
A. Beauville, {\it The L\"uroth problem}, 
Rationality problems in algebraic geometry, 1--27, 
Lecture Notes in Math., 2172, 
Fond. CIME/CIME Found. Subser., Springer, Cham, 2016. 
\bibitem[BCTSSD85]{BCTSSD85} 
A. Beauville, J.-L. Colliot-Th\'el\`ene, J.-J. Sansuc, P. Swinnerton-Dyer, 
{\it Vari\'et\'es stablement rationnelles non rationnelles}, (French) 
Ann. of Math. (2) \textbf{121} (1985) 283--318. 
\bibitem[BW20]{BW20} 
O. Benoist, O. Wittenberg, 
{\it The Clemens-Griffiths method over non-closed fields}, 
Algebr. Geom. \textbf{7} (2020) 696--721. 
\bibitem[BW23]{BW23} 
O. Benoist, O. Wittenberg, 
{\it Intermediate Jacobians and rationality over arbitrary fields}, 
Ann. Sci. \'Ec. Norm. Sup\'er. (4) \textbf{56} (2023) 1029--1086. 
%arXiv:1909.12668. 
\bibitem[Che54]{Che54} 
C. Chevalley, 
{\it On algebraic group varieties}, 
J. Math. Soc. Japan \textbf{6} (1954) 303--324. 
\bibitem[CG72]{CG72} 
C. H. Clemens, P. A. Griffiths, 
{\it The intermediate Jacobian of the cubic threefold}, 
Ann. of Math. (2) \textbf{95} (1972) 281--356. 
%\bibitem[CK00]{CK00} 
%A. Cortella, B. Kunyavskii, 
%{\it Rationality problem for generic tori in simple groups}, 
%J. Algebra \textbf{225} (2000) 771--793. 
\bibitem[CT07]{CT07} 
J.-L. Colliot-Th\'{e}l\`{e}ne, 
{\it Lectures on linear algebraic groups}, 
Beijing Lectures, Morning Side Centre, April 2007. 
Available from \url{https://www.imo.universite-paris-saclay.fr/~colliot/BeijingLectures2Juin07.pdf}
\bibitem[CTS77]{CTS77} 
J.-L. Colliot-Th\'{e}l\`{e}ne, J.-J. Sansuc, 
{\it La R-\'{e}quivalence sur les tores}, (French) 
Ann. Sci. \'{E}cole Norm. Sup. (4) \textbf{10} (1977) 175--229. 
\bibitem[CTS80]{CTS80} 
J.-L. Colliot-Th\'{e}l\`{e}ne, J.-J. Sansuc, 
{\it La descente sur les vari\'et\'es rationnelles}, (French) 
Journ\'ees de G\'eometrie Alg\'ebrique d'Angers, 
Juillet 1979/Algebraic Geometry, Angers, 1979, pp. 223--237, 
Sijthoff \& Noordhoff, Alphen aan den Rijn-Germantown, Md., 1980. 
\bibitem[CTS87]{CTS87} 
J.-L. Colliot-Th\'{e}l\`{e}ne, J.-J. Sansuc, 
{\it Principal homogeneous spaces under flasque tori: Applications}, 
J. Algebra \textbf{106} (1987) 148--205. 
\bibitem[CTS07]{CTS07} 
J.-L. Colliot-Th\'{e}l\`{e}ne, J.-J. Sansuc, 
{\it The rationality problem for fields of invariants under linear 
algebraic groups (with special regards to the Brauer group)}, 
Algebraic groups and homogeneous spaces, 113--186, 
Tata Inst. Fund. Res. Stud. Math., 19, Tata Inst. Fund. Res., Mumbai, 2007.
\bibitem[End11]{End11} 
S. Endo, 
{\it The rationality problem for norm one tori}, 
Nagoya Math. J. \textbf{202} (2011) 83--106.  
%\bibitem[EK17]{EK17} S. Endo, M. Kang, 
%{\it Function fields of algebraic tori revisited}, 
%Asian J. Math. \textbf{21} (2017) 197--224.
\bibitem[EM73]{EM73} 
S. Endo, T. Miyata, 
{\it Invariants of finite abelian groups}, 
J. Math. Soc. Japan \textbf{25} (1973) 7--26. 
\bibitem[EM75]{EM75} 
S. Endo, T. Miyata, 
{\it On a classification of the function fields of algebraic tori}, 
Nagoya Math. J. \textbf{56} (1975) 85--104. 
%\bibitem[EM80]{EM80} 
%S. Endo, T. Miyata, 
%{\it Corrigenda: On a classification of the function fields of algebraic tori 
%(Nagoya Math. J. \textbf{56} (1975) 85--104)}, 
%Nagoya Math. J. \textbf{79} (1980) 187--190.  
%\bibitem[Flo]{Flo} M. Florence, 
%{\it Non rationality of some norm-one tori}, preprint (2006).
%J. Algebra \textbf{177} (1995) 511--535. 
\bibitem[HHY20]{HHY20} 
S. Hasegawa, A. Hoshi, A. Yamasaki, 
{\it Rationality problem for norm one tori in small dimensions},
Math. Comp. \textbf{89} (2020) 923--940.
%\bibitem[Hos15]{Hos15} 
%A. Hoshi, 
%{\it On Noether's problem for cyclic groups of prime order}, 
%Proc. Japan Acad. Ser. A Math. Sci. \textbf{91} (2015) 39--44. 
\bibitem[Hos20]{Hos20} 
A. Hoshi, 
{\it Noether's problem and rationality problem for multiplicative 
invariant fields: a survey}, Algebraic number theory and related topics 2016, 
29--53, RIMS K\^oky\^uroku Bessatsu, \textbf{B77}, Res. Inst. Math. Sci. 
(RIMS), Kyoto, 2020. 
\bibitem[HKK14]{HKK14}
A. Hoshi, M. Kang, H. Kitayama, 
{\it Quasi-monomial actions and some $4$-dimensional rationality problems},
J. Algebra \textbf{403} (2014) 363--400. 
%\bibitem[HKY]{HKY} 
%A. Hoshi, M. Kang, A. Yamasaki, 
%{\it Class numbers and algebraic tori}, arXiv:1312.6738v2. 
\bibitem[HY17]{HY17} 
A. Hoshi, A. Yamasaki, 
{\it Rationality problem for algebraic tori}, 
Mem. Amer. Math. Soc. \textbf{248} (2017) no. 1176, v+215 pp. 
\bibitem[HY21]{HY21} 
A. Hoshi, A. Yamasaki, 
{\it Rationality problem for norm one tori},
Israel J. Math. \textbf{241} (2021) 849--867. 
%\bibitem[H\"{u}r84]{Hur84} 
%W. H\"{u}rlimann, 
%{\it On algebraic tori of norm type}, Comment. Math. Helv. 
%\textbf{59} (1984) 539--549. 
\bibitem[Isk67]{Isk67} 
V. A. Iskovskikh, 
{\it Rational surfaces with a pencil of rational curves}, 
Math. USSR Sb. \textbf{3} (1967) 563--587. 
\bibitem[Isk70]{Isk70} 
V. A. Iskovskikh, 
{\it Rational surfaces with a pencil of rational curves and with positive square of the canonical class}, 
Math. USSR Sb. \textbf{12} (1970) 91--117. 
\bibitem[Isk72]{Isk72} 
V. A. Iskovskikh, 
{\it Birational properties of a surface of degree $4$ in $\mathbb{P}_k^4$}, 
Math. USSR Sb. \textbf{17} (1972) 30--36. 
\bibitem[Isk97]{Isk97} 
V. A. Iskovskikh, 
{\it On the rationality problem for 
%algebraic threefolds. 
three-dimensional algebraic varieties}, 
(Russian) Tr. Mat. Inst. Steklova \textbf{218} (1997), 
Anal. Teor. Chisel i Prilozh., 190--232; 
translation in Proc. Steklov Inst. Math. \textbf{218} (1997) 186--227. 
\bibitem[IM71]{IM71}
V. A. Iskovskikh, Yu. I. Manin, 
{\it Three-dimensional quartics and counterexamples to the L\"roth problem}, 
(Russian) Mat. Sb. (N.S.) \textbf{86(128)} (1971) 140--166; 
translation in Math. USSR-Sb. \textbf{15} (1971) 141--166. 
%\bibitem[KZ20]{KZ20}
%M. Kang, J. Zhou, 
%{\it Noether's problem for some semidirect products}, 
%Adv. Math. \textbf{368} (2020) 107164, 21 pp. 
\bibitem[Kun07]{Kun07} 
B. E. Kunyavskii, 
{\it Algebraic tori --- thirty years after}, 
Vestnik Samara State Univ. (2007) 198--214. 
%\bibitem[LL00]{LL00} 
%N. Lemire, M. Lorenz, 
%{\it On certain lattices associated with generic division algebras}, 
%J. Group Theory \textbf{3} (2000) 385--405.
%\bibitem[LeB95]{LeB95} 
%L. Le Bruyn, {\it Generic norm one tori},
%Nieuw Arch. Wisk. (4) \textbf{13} (1995) 401--407. 
\bibitem[Len74]{Len74} 
H. W. Lenstra, Jr., 
{\it Rational functions invariant under 
a finite abelian group}, Invent. Math. \textbf{25} (1974) 299--325.
%\bibitem[Man74]{Man74} 
%Yu. I. Manin, 
%{\it Cubic forms: algebra, geometry, arithmetic}, 
%Translated from the Russian by M. Hazewinkel. 
%North-Holland Mathematical Library, Vol. 4. 
%North-Holland Publishing Co., Amsterdam-London; 
%American Elsevier Publishing Co., New York, 1974. vii+292 pp.
%North-Holland Mathematical Library 4, North-Holland Publishing Co., 
%Amsterdam, 1974. vii+292 pp.
\bibitem[Man86]{Man86} 
Yu. I. Manin, 
{\it Cubic forms: algebra, geometry, arithmetic},  
Second edition. North-Holland Mathematical Library 4, 
North-Holland Publishing Co., Amsterdam, 1986. x+326 pp.
\bibitem[MT86]{MT86}
Yu. I. Manin, M. A. Tsfasman, 
{\it Rational varieties: algebra, geometry, arithmetic}, (Russian)
Uspekhi Mat. Nauk \textbf{41} (1986) 43--94; 
translation in Russian Math. Surveys \textbf{41} (1986) 51--116. 
%\bibitem[Mas55]{Mas55} 
%K. Masuda, 
%{\it On a problem of Chevalley}, 
%Nagoya Math. J. \textbf{8} (1955) 59--63.
%\bibitem[Mas68]{Mas68} 
%K. Masuda, 
%{\it Application of the theory of the group of classes of 
%projective modules to the existence problem of independent parameters 
%of invariant}, 
%J. Math. Soc. Japan \textbf{20} (1968) 223--232. 
\bibitem[Mer17]{Mer17} 
A. S. Merkurjev, 
{\it Invariants of algebraic groups and retract rationality of classifying spaces}, 
Algebraic groups: structure and actions, 277--294. 
Proc. Sympos. Pure Math., 94, American Mathematical Society, Providence, RI, 2017. 
\bibitem[Ono61]{Ono61} T. Ono, 
{\it Arithmetic of algebraic tori}, 
Ann. of Math. (2) \textbf{74} (1961) 101--139. 
\bibitem[Ono63]{Ono63} 
T. Ono, 
{\it On the Tamagawa number of algebraic tori}, 
Ann. of Math. (2) \textbf{78} (1963) 47--73.
%\bibitem[Pla17]{Pla17} 
%B. Plans, 
%{\it On Noether's rationality problem for cyclic groups over $\mathbbm{Q}$}, 
%Proc. Amer. Math. Soc. \textbf{145} (2017) 2407--2409. 
\bibitem[Sal84]{Sal84} 
D. J. Saltman, 
{\it Retract rational fields and cyclic Galois extensions}, 
Israel J. Math. \textbf{47} (1984) 165--215.  
\bibitem[She04]{She04} 
N. I. Shepherd-Barron, 
{\it Stably rational irrational varieties}, 
The Fano Conference, 693--700, Univ. Torino, Turin, 2004. 
\bibitem[Swa69]{Swa69} 
R. G. Swan, {\it Invariant rational functions and a problem of Steenrod}, 
Invent. Math. \textbf{7} (1969) 148--158. 
\bibitem[Swa83]{Swa83} 
R. G. Swan, 
{\it Noether's problem in Galois theory}, 
Emmy Noether in Bryn Mawr (Bryn Mawr, Pa., 1982), 21--40, Springer, 
New York-Berlin, 1983. 
\bibitem[Swa10]{Swa10} 
R. G. Swan, 
{\it The flabby class group of a finite cyclic group}, 
Fourth International Congress of Chinese Mathematicians, 259--269, 
AMS/IP Stud. Adv. Math., 48, Amer. Math. Soc., Providence, RI, 2010. 
\bibitem[Voi16]{Voi16} 
C. Voisin, 
{\it Stable birational invariants and the L\"uroth problem}, 
Surveys in differential geometry 2016, 
Advances in geometry and mathematical physics, 313--342, 
Surv. Differ. Geom., 21, Int. Press, Somerville, MA, 2016.
\bibitem[Vos69]{Vos69} 
V. E. Voskresenskii, 
{\it The birational equivalence of linear algebraic groups}, (Russian) 
Dokl. Akad. Nauk SSSR \textbf{188} (1969) 978--981; 
erratum, ibid. 191 1969 nos., 1, 2, 3, vii; 
translation in Soviet Math. Dokl. \textbf{10} (1969) 1212--1215.
\bibitem[Vos70]{Vos70} 
V. E. Voskresenskii, 
{\it Birational properties of linear algebraic groups}, (Russian) 
Izv. Akad. Nauk SSSR Ser. Mat. \textbf{34} (1970) 3--19; 
translation in Math. USSR-Izv. \textbf{4} (1970) 1--17. 
%\bibitem[Vos70b]{Vos70b} 
%V. E. Voskresenskii, 
%{\it On the question of the structure of the subfield of invariants 
%of a cyclic group of automorphisms of the $\bQ(x_1,\ldots,x_n)$}, (Russian) 
%Izv. Akad. Nauk SSSR Ser. Mat. \textbf{34} (1970) 366--375; 
%translation in Math. USSR-Izv. \textbf{4} (1970) 371--380.
%\bibitem[Vos71]{Vos71} 
%V. E. Voskresenskii, {\it Rationality of certain algebraic tori}, 
%Izv. Akad. Nauk SSSR Ser. Mat. (Russian) \textbf{35} (1971) 1037--1046; 
%translation in Math. USSR-Izv. \textbf{5} (1971) 1049--1056. 
%\bibitem[Vos73]{Vos73} 
%V. E. Voskresenskii, {\it Fields of invariants of abelian groups}, 
%Uspekhi Mat. Nauk (Russian) \textbf{28} (1973) 77--102; 
%translation in Russian Math. Surveys \textbf{28} (1973) 79--105. 
\bibitem[Vos74]{Vos74} 
V. E. Voskresenskii, 
{\it Stable equivalence of algebraic tori}, (Russian) 
Izv. Akad. Nauk SSSR Ser. Mat. \textbf{38} (1974) 3--10; 
translation in Math. USSR-Izv. \textbf{8} (1974) 1--7.
\bibitem[Vos98]{Vos98} 
V. E. Voskresenskii, 
{\it Algebraic groups and their birational invariants}, 
Translated from the Russian manuscript by Boris Kunyavskii, 
Translations of Mathematical Monographs, 179. 
American Mathematical Society, Providence, RI, 1998.
\end{thebibliography}
\end{document}